\UseRawInputEncoding
\documentclass[12pt]{article}
\usepackage{amsmath,amssymb,amsthm}
\usepackage{graphicx}
\usepackage{subfigure}
\usepackage{amsmath,amsfonts,amssymb,amscd}
\usepackage{indentfirst,graphicx,epstopdf}
\usepackage{graphicx,psfrag,ifpdf,enumerate}
\usepackage{caption}
\usepackage{amsthm, amsfonts, color}
\usepackage[bookmarksnumbered, plainpages]{hyperref}
\usepackage[numbers,sort&compress]{natbib}
\usepackage[noend]{algpseudocode}
\usepackage{algorithmicx,algorithm}
\usepackage{psfrag}
\usepackage{color}
\usepackage{enumerate}
\usepackage{epstopdf}
\input{epsf}

\newcommand\tabcaption{\def\@captype{table}\caption}

\newtheorem{theorem}{Theorem}

\newtheorem{lemma}{Lemma}

\newtheorem{prop}[theorem]{Proposition}
\theoremstyle{definition}
\newtheorem{rem}{Remark}
\newtheorem{prob}{Problem}

\newtheorem{exam}{Example}

\begin{document}
\title{\Large \bf Counting rainbow triangles in edge-colored graphs}

\author{Xueliang Li\thanks{Center for Combinatorics and LPMC,
Nankai University, Tianjin 300071, China.
Email: lxl@nankai.edu.cn.},~~Bo Ning\thanks{Corresponding author. College of Cyber Science,
Nankai University, Tianjin 300350, China.
Email: bo.ning@nankai.edu.cn.},~~Yongtang Shi\thanks{Center for Combinatorics and LPMC,
Nankai University, Tianjin 300071, China. Email: shi@nankai.edu.cn.},~~Shenggui
Zhang\thanks{$^{a}$School of Mathematics and Statistics,
Northwestern Polytechnical University, Xi'an, Shaanxi 710129, China;
$^{b}$Xi'an-Budapest Joint Research Center for Combinatorics,
Northwestern Polytechnical University, Xi'an, Shaanxi 710129, China.
Email: sgzhang@nwpu.edu.cn.}}
\maketitle
\begin{abstract}
Let $G$ be an edge-colored graph on $n$ vertices.
The minimum color degree of $G$, denoted by $\delta^c(G)$,
is defined as the minimum number of colors assigned to
the edges incident to a vertex in $G$. In 2013,
H. Li proved that an edge-colored graph $G$
on $n$ vertices contains a rainbow triangle if $\delta^c(G)\geq \frac{n+1}{2}$.
In this paper, we obtain several estimates on the number of rainbow triangles
through one given vertex in $G$. As consequences,
we prove counting results for rainbow triangles
in edge-colored graphs. One main theorem states that
the number of rainbow triangles in $G$ is at least
$\frac{1}{6}\delta^c(G)(2\delta^c(G)-n)n$,
which is best possible by considering the rainbow $k$-partite
Tur\'an graph, where its order is divisible by $k$.
This means that there are $\Omega(n^2)$
rainbow triangles in $G$ if $\delta^c(G)\geq \frac{n+1}{2}$,
and $\Omega(n^3)$ rainbow triangles in $G$ if
$\delta^c(G)\geq cn$ when $c>\frac{1}{2}$. Both results are tight
in sense of the order of the magnitude. We also prove a counting
version of a previous theorem on rainbow triangles
under a color neighborhood union condition due to Broersma et al., and
an asymptotically tight color degree condition forcing a colored friendship subgraph $F_k$ (i.e.,
$k$ rainbow triangles sharing a common vertex).
\end{abstract}

\section{Introduction}\label{Sec:1}
Throughout this paper, we only consider finite undirected simple graphs.
Let $G$ be a graph. By an edge-coloring of $G$, we mean a function
$C: E\rightarrow \mathbb{N}$, where $\mathbb{N}$ is the set of non-negative
integers. If $G$ has such an edge-coloring, we call $G$ an edge-colored
graph and denote it by $(G,C)$. For a vertex $v\in V(G)$, the color neighborhood $CN_G(v)$
is defined as the set $\{C(e): {e~is~incident~with~v}\}$,
and the \emph{color degree} of $v$ is denoted by $d_G^c(v):=|CN_G(v)|$.
We denote by $\delta^c(G):=\min\{d_G^c(v):v\in V(G)\}$,
and by $c(G)$ the number of colors appearing on $E(G)$.
Let $\overline{\sigma}^c_2(G)=\min\{d^c(x)+d^c(y):xy\in E(G)\}.$
For a vertex $v\in V(G)$, the \emph{monochromatic degree} of $v$ (in $G$),
denoted by $d_G^{mon}(v)$, is defined as the maximum number
of edges incident to $v$ colored with a same color.
A subgraph $H$ of $G$ is called \emph{properly-colored}
if every two incident edges are assigned with different
colors, and is called \emph{rainbow}
if all of its edges have distinct colors.
When there is no possibility of confusion, we will drop the subscript
$G$. For example, we use $\delta^c$ instead of $\delta^c(G)$.
For notation and terminology not defined here, we refer to Bondy
and Murty \cite{BM08}.

Rainbow and properly-colored subgraph problems have received much attention from graph theorists,
see \cite{BLWZ05,CKRY16,FLZ18,ADH19,CMNO21}. For surveys, see \cite{KL08,C18}.
In 2013, H. Li \cite{L13} proved a minimum color degree condition for rainbow triangles, which was
conjectured in \cite{LW12}.
\begin{theorem}[\cite{L13}]\label{Thm:Li}
Let $(G,C)$ be an edge-colored graph on $n\geq 3$ vertices. If $\delta^c(G)\geq \frac{n+1}{2}$
then $G$ contains a rainbow triangle.
\end{theorem}

A slightly stronger Dirac-type result was proved by
B. Li, Ning, Xu, and Zhang in \cite{LNXZ14}.

\begin{theorem}[\cite{LNXZ14}]
Let $(G,C)$ be an edge-colored graph on $n\geq 5$ vertices. If $\delta^c(G)\geq \frac{n}{2}$
then $G$ contains a rainbow triangle unless $G$ is a properly colored $K_{\frac{n}{2},\frac{n}{2}}$.
\end{theorem}
Theorem \ref{Thm:Li} motivated much attention on rainbow subgraphs.
Czygrinow, Molla, Nagle, and Oursler \cite{CMNO21} recently proved
that the same condition in Theorem \ref{Thm:Li} ensures a rainbow $\ell$-cycle
$C_{\ell}$ whenever $n>432\ell$, which is sharp for a fixed odd integer
$\ell\geq 3$ when $n$ is sufficiently large. The authors in \cite{LNXZ14}
proposed a new type condition, i.e., every edge-colored graph $(G,C)$ on $n$ vertices
satisfying $e(G)+c(G)\geq \frac{n(n+1)}{2}$ contains a rainbow triangle,
where $e(G)$ is the number of edges in $G$ and $c(G)$ is the number of
all colors appearing on $E(G)$. This motivated further studies on rainbow cliques \cite{XHWZ16}
and properly-colored $C_4$'s \cite{XMZ20}.

The original purpose of this article is to study the supersaturation
problem of rainbow triangles in edge-colored graphs. This problem is
obviously motivated by the study of
supersaturation problem
of triangles in graphs. It studies the following function:
for triangle $C_3$ and for integers $n,t\geq1$,
$$h_{C_3}(n,t) = min\{t(G): |V(G)|=n, |E(G)|=ex(n,C_3)+t\},$$
where $t(G)$ is the number of $C_3$ in $G$ and $ex(n,C_3)$ is the Tur\'an function of $C_3$.
Improving Mantel's theorem,
Rademacher (unpublished, see \cite{E55}) proved that $h_{C_3}(n,1)\geq \lfloor\frac{n}{2}\rfloor$.
Erd\H{o}s \cite{E62-1,E62-2} proved that $h_{C_3}(n,k)\geq k\lfloor\frac{n}{2}\rfloor$
where $k\leq cn$ for some constant $c$. In fact, Erd\H{o}s conjectured that
$h_{C_3}(n,k)\geq k\lfloor\frac{n}{2}\rfloor$ for all $k<\lfloor\frac{n}{2}\rfloor$, which
was finally resolved by Lov\'asz and Simonovits \cite{LS83}.

One can ask for a rainbow analog of the above Erd\H{o}s'
conjecture. In this direction, answering an open
problem in \cite{FNXZ19}, Ehard and Mohr \cite{EM20}
proved there are at least $k$ rainbow triangles in an edge-colored
graph $(G,C)$ such that $e(G)+c(G)\geq \binom{n+1}{2}+k-1$.
If we consider $e(G)+c(G)$ as a variant of Tur\'an function
in edge-colored graphs, then the theorem above tells us that the
supersaturation phenomenon of rainbow triangles under this type of
condition is quite different from the original one. On the other hand,
the problem of finding a counting version of Theorem \ref{Thm:Li}
is still open.

We denote by $\mathcal{G}^{*}_n$ the family of
edge-colored graphs on $n$ vertices with the minimum
color degree at least $\frac{n+1}{2}$,
by $rt(G)$ the number of rainbow triangles in
an edge-colored graph $G$, and by $rt(G;v)$ be the number of
rainbow triangles through a vertex $v$ in $G$.
Denote by $$f(n):=\min\{rt(G):G\in \mathcal{G}^{*}_n\}.$$

Proving a special case of a conjecture which states that
every edge-colored graph on $n\geq 20$ vertices contains
two disjoint rainbow triangles if the minimum color
degree is at least $\frac{n+2}{2}$, Hu, Li, and Yang
developed a key lemma \cite[Lemma~1]{HLY20}, from which
one can easily obtain $f(n)=\Omega(n)$.
One may dare to guess that $f(n)=\Omega(n^2)$.
Our first humble contribution confirms this.
\begin{rem}
Throughout this paper, we repeatedly assume that an edge-colored
graph $(G,C)$ satisfying $\delta^c(G)\geq \frac{n+1}{2}$ and subject
to this, $e(G)$ is minimal. Here the word ``minimal" means
that deleting any edge in $G$ will decrease the color degree
of some vertex of $G$. That is, $G$ contains no monochromatic
$C_3$ or $P_4$. Furthermore, we can see that a spanning subgraph
of $G$ with a same color should be a star forest.
\end{rem}

\begin{theorem}\label{Thm:Main1}
Let $(G,C)$ be an edge-colored graph on $n$ vertices.
Suppose that $\delta^c(G)\geq \frac{n+1}{2}$, and subject to this,
$e(G)$ is minimal.
Then
$$rt(G)\geq\frac{e(G)(\overline{\sigma}^c_2(G)-n)}{3}+\frac{1}{6}\sum_{v\in V(G)}(n-d(v)-1)(d(v)-d^c(v)).$$
\end{theorem}

As consequences of Theorem \ref{Thm:Main1},
we obtain two counting versions of Theorem \ref{Thm:Li}.

\begin{theorem}\label{Thm:Rt1}
Let $(G,C)$ be an edge-colored graph on $n$ vertices.
Then $$rt(G)\geq \frac{1}{6}\delta^c(G)(2\delta^c(G)-n)n.$$
In particular, if $\delta^c(G)>cn$ for $c>\frac{1}{2}$, then
$$rt(G)\geq\frac{c(2c-1)}{6}n^3.$$
\end{theorem}

One may wonder the tightness of Theorem \ref{Thm:Rt1}.
The following example shows that Theorem \ref{Thm:Rt1}
is the best possible.
\begin{exam}

Let $G$ be a rainbow $k$-partite Tur\'an graph on $n$ vertices where $k|n$ and $k\geq 3$. Then
there are exactly $\binom{k}{3}(\frac{n}{k})^3=\frac{(k-1)(k-2)}{6k^2}n^3$
rainbow triangles. By Theorem \ref{Thm:Rt1}, there
are at least $\binom{k}{3}(\frac{n}{k})^3=\frac{(k-1)(k-2)}{6k^2}n^3$ rainbow triangles.
\end{exam}

Setting $\delta^c(G)=\frac{n+1}{2}$ in
Theorem \ref{Thm:Rt1}, we obtain
the right hand of the following.

\begin{theorem}\label{Thm:Rt2}
For even $n\geq 4$, we have $\frac{n^2}{4}\geq f(n)\geq\frac{n^2+2n}{6};$ for odd
$n\geq 3$, we have $\frac{n^2-1}{8}\geq f(n)\geq \frac{n^2+n}{12}$.
\end{theorem}
For Theorem \ref{Thm:Rt2}, the leftmost of each inequality (for $f(n)$) of
Theorem \ref{Thm:Rt2} was shown by
the following two examples. From Theorem \ref{Thm:Rt2},
we infer $f(n)=\Theta(n^2)$.

\begin{exam}
Let $(G,C)$ be a rainbow graph of order $n$ where $n$
is divisible by 4. Let
$V(G)=X_1\cup X_2$, $|X_1|=|X_2|=\frac{n}{2}$, and
each of $G[X_1]$ and $G[X_2]$ consists of a perfect matching of size
$\frac{n}{4}$. In addition, $G-E(X_1)-E(X_2)$ is balanced and complete bipartite.
For each edge $e\in E(X_1)$, it is contained in exactly
$\frac{n}{2}$ rainbow triangles. So does each edge in $G[X_2]$.
Therefore, there are exactly $\frac{n^2}{4}$ rainbow triangles in $G$.
\end{exam}

\begin{exam}
Let $(G,C)$ be a rainbow graph of order $n$ where $n\equiv1\pmod4$.
Let
$V(G)=X_1\cup X_2$, $|X_1|=\frac{n+1}{2}$ and $|X_2|=\frac{n-1}{2}$, and
$G[X_1]$ consists of a perfect matching of size
$\frac{n+1}{4}$. In addition, $G-E(X_1)$ is complete bipartite.
For each edge $e\in E(X_1)$, it is contained in exactly
$\frac{n-1}{2}$ rainbow triangles and so does each edge in $G[X_1]$.
Therefore, there are exactly $\frac{n^2-1}{8}$ rainbow triangles in $G$.
\end{exam}

In 2005, Broersma, X. Li, Woeginger, and Zhang \cite{BLWZ05} proved
an edge-colored graph $(G,C)$ on $n\geq 4$ vertices contains a rainbow $C_3$ or
a rainbow $C_4$ if $|CN(u)\cup CN(v)|\geq n-1$ for every pair of vertices
$u,v\in V(G)$.
Define $G$ to be a rainbow $K_{\frac{n}{2},\frac{n}{2}}$ where
$n$ is even. Then $|CN(u)\cup CN(v)|=n-1$ for each pair
of vertices $u$ and $v$, and $G$ contains no rainbow triangles.
Thus, one need slightly enhance the color degree condition
when finding rainbow triangles.
Broersma et al.'s theorem  was
generalized by Fujita, Ning, Xu
and Zhang \cite{FNXZ19} to the one forcing rainbow
triangles under the same condition.

In this paper, we extend both theorems mentioned
to a counting version as follows.

\begin{theorem}\label{Thm:Colorneighborcounting}
Let $(G,C)$ be an edge-colored graph of order $n\geq 4$
such that $|CN(u)\cup CN(v)|\geq n$ for every pair of vertices
$u,v\in V(G)$. Then $G$ contains $\frac{n^2-2n}{24}$ rainbow $C_3$'s.
\end{theorem}

We also prove some better estimate on the number of rainbow triangles
through vertices with high monochromatic degree.

\begin{theorem}\label{Thm:Main2}
Let $(G,C)$ be an edge-colored graph on $n$ vertices with $\delta^c(G)$
and furthermore, $e(G)$ is minimal. Let $V(G)=\{v_1,v_2,\ldots,v_n\}$
such that
$d^{mon}_G(v_1)\geq d^{mon}_G(v_2)\geq \cdots \geq d^{mon}_G(v_n).$
Then for each $1\leq k\leq \delta^c(G)-1$,
\begin{align*}
\sum_{i=1}^k rt(G;v_i)&\geq\frac{1}{2}\left(\sum_{i=1}^k d^{mon}_G(v_i)+k(\delta^c(G)-1)\right)(\overline{\sigma}_2^c(G)-n)+\frac{\Delta_k(G)}{2}.
\end{align*}
where
$$\Delta_k(G)=\left(\delta^c(G)\sum_{i=1}^{k}d^{mon}_G(v_i)-k\sum_{i=1}^{\delta^c(G)}d^{mon}_G(v_i)\right).$$
\end{theorem}

The above theorem has the following simple but useful corollary.
\begin{theorem}
Let $(G,C)$ be an edge-colored graph on $n$ vertices with $\delta^c(G)\geq \frac{n+1}{2}$.
Let $V(G)=\{v_1,v_2,\ldots,v_n\}$ such that
$d^{mon}_G(v_1)\geq d^{mon}_G(v_2)\geq \cdots \geq d^{mon}_G(v_n).$
Then for each $1\leq k\leq \delta^c(G)-1$,
\begin{align*}
\sum_{i=1}^k rt(G;v_i)\geq\frac{k\delta^c(G)}{2}.
\end{align*}
\end{theorem}

The \emph{friendship graph} $F_k$ is a graph
consisting of $k$ triangles sharing a common vertex.
Finally, we obtain some color degree condition for
the existence of some rainbow triangles sharing one common vertex,
i.e., the underlying graph is a friendship subgraph.
This extends Theorem 1 in another way.
\begin{theorem}\label{Thm:FriendshipSub}
Let $k\geq 2$ and $n\geq 50k^2$.
Let $(G,C)$ be an edge-colored graph on $n$ vertices.
If $\delta^c(G)\geq \frac{n}{2}+k-1$ then $G$
contains $k$ rainbow triangles sharing one common vertex.
\end{theorem}

This paper is organised as follows.
In Section \ref{Sec:2}, we prove one estimate on the
number of rainbow triangles through one given vertex.
As consequences, we prove Theorems \ref{Thm:Main1}
and \ref{Thm:Colorneighborcounting} (together with
some corollaries). In Section \ref{Sec:3}, we prove another estimate
on the number of rainbow triangles through vertices with high monochromatic
degree, and prove Theorem \ref{Thm:Main2}. In Section \ref{Sec:4},
we prove a theorem slightly stronger than Theorem \ref{Thm:FriendshipSub}.
We conclude this paper with some open problems in the last section.
\section{Counting rainbow triangles}\label{Sec:2}
The main purpose of this section is to prove Theorem \ref{Thm:Main1}.
We first prove the following lemma, whose proof is partly
inspired by \cite[Lemma~1]{HLY20}.
\begin{lemma}\label{Lemma:Countingrt-1}
Let $(G,C)$ be an edge-colored graph on $n$ vertices with $\delta^c(G)\geq \frac{n+1}{2}$
and furthermore, $e(G)$ is minimal.
Then for each $v\in V(G)$,
\begin{align*}
&rt(G;v)\\
&\geq\frac{1}{2}\left((n-d(v)-1)(d(v)-d^c(v)+\sum_{1\leq j\leq d^c(v)}\sum_{a\in N_j(v)}(d_j(v)-d_j(a)+\sum_{a\in N_v}(d^c(v)+d^c(a)-n)\right).
\end{align*}
\end{lemma}

\begin{proof}
Since $G$ is edge-minimal, there is no monochromatic path of length 3
and no monochromatic triangle in $G$.

Let $N_v:=N_G(v)$. Without loss of generality, assume that
$CN_G(v)=\{1,2,\ldots,s\}$, where $s=d^c(v)$.
Let $N_{j}(v):=\{u:C(uv)=j, v\in N_v\}$ and $d_j(v):=|N_j(v)|$, where $1\leq j\leq s$.
Furthermore, assume that $d_1(v)\geq d_2(v)\geq \cdots \geq d_s(v)$.
So $d^{mon}(v)=d_1$.

For the vertex $v\in V(G)$, define a digraph $D_v$ on $N_v$
as follows:
$\overrightarrow{ab}\in A(D_v)$ if and only if $ab\in E(G)$
and $C(ab)\neq C(va)$, i.e., $vab$ is a rainbow path of length 2.
Therefore, for any two vertices $x,y\in N_j(v)$ (if $|N_j(v)|\geq 2$),
there is either a 2-cycle $xyx$ or no arc between $x$ and $y$
(since otherwise, there is a monochromatic $C_3$, a contradiction).

For $a\in N_v$, let $S_a\subseteq N_v\backslash \{a\}$ be maximal
such that $C(au),C(au')$ and $C(av)$ are pairwise different for any two distinct
vertices $u,u'\in N_v$. According to the definition of $D_v$,
every edge $au$, $u\in S_a$, corresponds to an out-arc from $a$
in $D_v$.
Notice that
\begin{align*}
d^+_{D_v}(a)&\geq |S_a|\geq CN_{G[N_v\cup\{v\}]}(a)\geq d^c(a)-1-|V(G)\setminus (N_v\cup \{v\})|.
\end{align*}
Thus, we have
\begin{align*}
d^+_{D_v}(a)\geq d^c(a)+d_G(v)-n.
\end{align*}
Therefore,
\begin{align}\label{Lem1align:e1}
\sum_{a\in N_v}d^+_{D_v}(a)&\geq d_G(v)(d_G(v)-n)+\sum_{a\in N_v}d^c(a)\nonumber\\
&=d_G(v)\left(\sum_{j=1}^s(d_j-1)\right)+\sum_{a\in N_v}(d^c(a)+d^c(v)-n).\nonumber\\
\end{align}

Next, consider $\sum\limits_{a\in N_v}d^-_{D_v}(a)$.
Note that for any vertices $a, b\in N_j(v)$ such
that $ab\in E(G)$ ($d_j(v)\geq 2$), if $\overrightarrow{ba}\in A(D_v)$
then $c(va)=c(ab)$, and so $ba$ is not contained in a 2-cycle.
Hence for every $j\in [1,s]$ and every $a\in N_j(v)$,
there are at most $d_j(a)-1$ in-arcs which are not contained in 2-cycles,
such that $a$ is a common sink. Observe that for any such vertex $a\in N_v$,
if $a\in N_j$ then $d_j=1$; otherwise there is a monochromatic $P_3$ in $G$,
a contradiction.

For $1\leq j\leq s$, let $n_j$ be the number of 2-cycles in $D_v[N_j]$.
Let $n_0$ be the number of all 2-cycles $xyx$ in $D_v$ such that $C(xv)\neq C(yv)$.
That is, $n_0=rt(G;v)$.

Thus, we have
\begin{equation}\label{Lem1align:e2}
\sum_{a\in N_v}d^-_{D_v}(a)\leq \sum_{1\leq j\leq d^c(v),d_j(v)=1}\sum_{a\in N_j(v)}(d_j(a)-1)+2n_0+2\sum_{j=1}^sn_j.
\end{equation}
Since $$2n_j\leq d_j(v)(d_j(v)-1),$$
From (\ref{Lem1align:e2}), we can obtain that
\begin{align}\label{Lem1align:e3}
\sum_{a\in N_v}d^-_{D_v}(a)\leq \sum_{1\leq j\leq d^c(v),d_j(v)=1}\sum_{a\in N_j(v)}(d_j(a)-1)+2n_0+\sum_{j=1}^sd_j(v)(d_j(v)-1).
\end{align}

As $$\sum_{a\in N_v}d^+_{D_v}(a)=\sum_{a\in N_v}d^-_{D_v}(a),$$
combining (\ref{Lem1align:e1}) and (\ref{Lem1align:e3}),
we have
\begin{align}\label{Lem1align:e4}
2n_0&\geq \sum_{a\in N_v}(d^c(v)+d^c(a)-n)+d(v)\left(\sum_{j=1}^s(d_j-1)\right)-2\sum_{j=1}^sd_j(v)(d_j(v)-1)\nonumber\\
&-\sum_{1\leq j\leq s,d_j(v)=1}\sum_{a\in N_j(v)}(d_j(a)-1)+\sum_{j=1}^sd_j(v)(d_j(v)-1).
\end{align}
Set
\begin{align*}
A=d(v)\left(\sum_{j=1}^s(d_j(v)-1)\right)-2\sum_{j=1}^sd_j(v)(d_j(v)-1),
\end{align*}
and
\begin{align*}
B=-\sum_{1\leq j\leq s,d_j(v)=1}\sum_{a\in N_j(v)}(d_j(a)-1)+\sum_{j=1}^sd_j(v)(d_j(v)-1).
\end{align*}
Then  (\ref{Lem1align:e4}) is equivalent to the following
\begin{align}\label{Lem1align:e5}
2n_0&\geq \sum_{a\in N_v}(d^c(v)+d^c(a)-n)+A+B.
\end{align}
By simple algebra,
\begin{align*}
A=\sum_{j=1}^s(d(v)-2d_j(v))(d_j(v)-1)=\sum_{j=1,d_j(v)\geq 2}^s(d(v)-2d_j(v))(d_j(v)-1).
\end{align*}
As $$d_1(v)\leq d(v)-d^c(v)+1$$ and $$d^c(v)\geq \frac{n+1}{2},$$ we have
$$
d(v)-2d_j(v)\geq d(v)-2d_1(v)\geq d(v)-2(d(v)-d^c(v)+1)=2d^c(v)-d(v)-2\geq n-d(v)-1,
$$
and so
\begin{align}\label{Lem1align:e6}
A&\geq \sum_{j=1}^s(n-d(v)-1)(d_j(v)-1)=(n-d(v)-1)(d(v)-d^c(v)).
\end{align}
Furthermore, we obtain
\begin{align}\label{Lem1align:e7}
B&=-\sum_{1\leq j\leq s,d_j(v)=1}\sum_{a\in N_j(v)}(d_j(a)-d_j(v))+\sum_{1\leq j\leq s, d_j(v)\geq 2}\sum_{a\in N_j(v)}(d_j(v)-d_j(a))\nonumber\\
&=\sum_{1\leq j\leq d^c(v)}\sum_{a\in N_j(v)}(d_j(v)-d_j(a)),
\end{align}
where $d_j(a)=1$ when $d_j(v)\geq 2$, since $G$ contains no monochromatic $P_4$.

Now, together with (\ref{Lem1align:e5}), (\ref{Lem1align:e6}), and (\ref{Lem1align:e7}),
we infer
\begin{align*}
2n_0&\geq \left(\sum_{a\in N_v}(d^c(v)+d^c(a)-n)\right)+(n-d(v)-1)(d(v)-d^c(v))\\
&+\sum_{1\leq j\leq d^c(v)}\sum_{a\in N_j(v)}(d_j(v)-d_j(a)).
\end{align*}
This proves Lemma \ref{Lemma:Countingrt-1}.
\end{proof}

By simple technique of counting in two ways, we have the following.
\begin{prop}\label{Prop:Section3}\footnote{Note that in the proof of Lemma  \ref{Lemma:Countingrt-1},
without loss of generality, we assume that $I_v=\{1,2,\ldots,d^c(v)\}$ for simply. In fact, for distinct vertices $u,v\in V(G)$,
$I_u$ may be not equal to $I_v$, and may be not a subset of $[1,C(G)]$.}
Let $(G,C)$ be an edge-colored graph with vertex set $V(G)$
and $\delta^c(G)\geq \frac{n+1}{2}$,
and furthermore, $e(G)$ is minimal. Let $I_v=\{C(uv):uv\in E(G)\}$.
For $k\in I_v$, $N_{k}(v):=\{u\in N_v: C(uv)=k\}$ and $d_{k}(v):=|N_k(v)|$.
Then
\begin{align}
\sum_{v\in V(G)}\sum_{k\in I_v}\sum_{a\in N_{k}(v)}(d_{k}(v)-d_{k}(a))=0.
\end{align}
\end{prop}

\begin{proof}
By definition of $N_k(v)$, we can see
\begin{align*}
\sum_{k\in I_v}\sum_{a\in N_k(v)}(d_k(v)-d_k(a))=\sum_{a\in N_G(v)} d_{C(va)}(v)-d_{C(va)}(a).
\end{align*}
By counting in two ways, we have
\begin{align*}
&\sum_{v\in V(G)}\sum_{a\in N_G(v)} (d_{C(va)}(v)-d_{C(va)}(a))\\
&=\sum_{xy\in E(G)} (d_{C(xy)}(x)-d_{C(xy)}(y))+(d_{C(xy)}(y)-d_{C(xy)}(x))=0.
\end{align*}
This proves Proposition \ref{Prop:Section3}.
\end{proof}

Now we can obtain one main result in this paper, which is stronger than Theorem
\ref{Thm:Main1}.
\begin{theorem}
Let $(G,C)$ be an edge-colored graph with vertex set $V(G)$.
Let  $n=|V(G)|$.
If $\delta^c\geq \frac{n+1}{2}$, and furthermore,
$e(G)$ is minimal, then we have,
\begin{align}
rt(G)&\geq\frac{1}{6}\sum_{v\in V(G)}\left((n-d(v)-1)(d(v)-d^c(v))+\sum_{a\in N_G(v)}(d^c(v)+d^c(a)-n)\right).\nonumber
\end{align}
\end{theorem}

\begin{proof}
The theorem follows from Lemma \ref{Lemma:Countingrt-1},
Proposition \ref{Prop:Section3}, and the fact that
$3rt(G)=\sum_{v\in V(G)}rt(G;v).$
\end{proof}
Finally, we present the proof of Theorem \ref{Thm:Colorneighborcounting}.

\noindent
{\bf Proof of Theorem \ref{Thm:Colorneighborcounting}.}
If $\delta^c\geq \frac{n+1}{2}$, then by Theorem \ref{Thm:Rt2},
$G$ contains $\frac{n^2+n}{12}$ rainbow $C_3$'s. Thus, $\delta^c\leq\frac{n}{2}$.
Choose $v\in V(G)$ such that $d_G^c(v)=\delta^c\leq \frac{n}{2}$. Set $G'=G-v$.

First we furthermore suppose that $d_G^c(v)\leq \frac{n-1}{2}$.
For a vertex $u$ adjacent to $v$, $|CN(u)\cup CN(v)|\geq n$. It follows
that
$$
d_G^c(u)+d_G^c(v)=|CN(u)\cup CN(v)|+|CN(u)\cap CN(v)|\geq n+1.
$$
It follows that $d_G^c(u)\geq \frac{n+3}{2}$. For a vertex $u$ non-adjacent
to $v$, we also have
$$
d_G^c(u)+d_G^c(v)=|CN(u)\cup CN(v)|+|CN(u)\cap CN(v)|\geq n.
$$
Thus, $d_G^c(u)\geq \frac{n+1}{2}$. It follows that $d_{G'}(u)\geq \frac{n+1}{2}>\frac{|G'|+1}{2}$.
Then by Theorem \ref{Thm:Rt1},
we have $$rt(G')\geq \frac{1}{6}\cdot\frac{n+1}{2}\left(2\cdot\frac{n+1}{2}-(n-1)\right)n\geq \frac{n^2+n}{6}.$$
So $d^c(v)=\frac{n}{2}$, i.e., $\delta^c=\frac{n}{2}$. In this case, for an edge $uv\in E(G)$,
$$d^c(u)+d^c(v)=|CN(u)\cup CN(v)|+|CN(u)\cap CN(v)|\geq n+1.$$
By setting $k=\delta^c-1$ in Theorem \ref{Thm:Main2}, we have
$$rt(G)\geq \frac{1}{3}\sum_{i=1}^{\delta^c-1}rt(G;v_i)\geq\frac{n^2-2n}{24},$$
where $V(G)=\{v_1,v_2,\ldots,v_n\}$ such that $d^{mon}(v_1)\geq d^{mon}(v_2)\geq \cdots\geq d^{mon}(v_n)$.
This proves the theorem.
$\hfill\blacksquare$

\section{Rainbow triangles through a specified vertex}\label{Sec:3}
The main purpose of this section is to obtain better estimate
of the number of rainbow triangles through a specified vertex
when the monochromatic degree of this vertex is large.
Before the proof, we need introduce some additional notations.

For a vertex $v\in V(G)$, let $X_v$ be the maximal subset of
$N_G(v)$ such that $c(va)=c(vb)$ for any two distinct vertices
$a,b\in X_v$.
Then $|X_v|=d^{mon}(v)$. Let $Y_v\subseteq (N_G(v)\setminus X_v)$
such that $c(va)\neq c(vb)$ for any two vertices $a,b\in Y_v$.
Thus, we have $|Y_v|\leq d^c(v)-1$.
In the following, set $$f(v):=\min\{d^c(u)+|Y_v|+1:u\in X_v\cup Y_v\}.$$

We first prove the following lemma, whose proof is a variant of Lemma \ref{Lemma:Countingrt-1}.
\begin{lemma}\label{Lemma:Countingrt-2}
Let $(G,C)$ be an edge-colored graph on $n$ vertices with $\delta^c(G)$,
and subject to this, $e(G)$ is minimal.
Then for each $v\in V(G)$,
\begin{align}
rt(G;v)\geq \frac{1}{2}\left((d^{mon}_G(v)+|Y_v|)(f(v)-n)+(|Y_v|d^{mon}_G(v)-\sum_{a\in Y_v}d^{mon}_G(a))\right).
\end{align}
\end{lemma}

\begin{proof}
For a vertex $v\in V(G)$, define a digraph $D_v$ on $X_v\cup Y_v$
as follows:
$\overrightarrow{ab}\in A(D_v)$ if and only if $ab\in E(G)$
and $c(ab)\neq c(va)$.
Let $n_1$ be the number of 2-cycles in $D_v[X]$.
Let $n_2$ be the number of other 2-cycles in $D_v$.
Apparently, $rt(v)\geq n_2$.

For $a\in X_v\cup Y_v$, let $S\subset (X_v\cup Y_v)\backslash \{a\}$ be maximal
such that $c(au),c(au'),c(av)$ are pairwise different for two distinct
vertices $u,u'\in X_v\cup Y_v$.
According to the definition of $D_v$,
every edge $au$, $u\in S$ gives an out-arc of $a$
in $D_v$.
Hence, we have
\begin{align*}
d^+_{D_v}(a)&\geq d^c(a)-1-|V(G)\setminus (X_v\cup Y_v\cup \{v\})|\\
&\geq f(v)+d^{mon}(v)-n-1.
\end{align*}
Therefore,
\begin{equation}\label{e1}
\sum_{a\in X_v\cup Y_v}d^+_{D_v}(a)\geq (d^{mon}(v)+|Y_v|)(f(v)+d^{mon}(v)-n-1).
\end{equation}

Next, consider $\sum_{a\in X_v\cup Y_v}d^-_{D_v}(a)$.
By reasoning the proof of Lemma \ref{Lemma:Countingrt-1} and a similar analysis, we obtain
\begin{equation}\label{e2}
\sum_{a\in X_v\cup Y_v}d^-_{D_v}(a)\leq \sum_{a\in Y_v}d^{mon}(a)-|Y_v|+2(n_1+n_2).
\end{equation}
Since $$\sum_{a\in X_v\cup Y_v}d^+_{D_v}(a)=\sum_{a\in X_v\cup Y_v}d^-_{D_v}(a)$$
and $$2n_1\leq d^{mon}(v)(d^{mon}(v)-1),$$
combining (\ref{e1}) and (\ref{e2}),
we have
$$\begin{aligned}
2n_2&\geq d^{mon}(v)(f(v)-n)+|Y_v|(f(v)+d^{mon}(v)-n)-\sum_{a\in Y_v}d^{mon}(a)\\
&=(d^{mon}(v)+|Y_v|)(f(v)-n)+(|Y_v|d^{mon}(v)-\sum_{a\in Y_v}d^{mon}(a)).
\end{aligned}$$
Hence,
$$rt(G;v)\geq n_2
\geq\frac{1}{2}\left((d^{mon}(v)+|Y_v|)(f(v)-n)+(|Y_v|d^{mon}(v)-\sum_{a\in Y_v}d^{mon}(a))\right).$$
The proof is complete.
\end{proof}

Specially, set $|Y_v|=d^c(v)-1$.
Then $$f(v):=\min\{d^c(v)+d^c(u):u\in X_v\cup Y_v\}\geq \overline{\sigma}^c_2(G),$$
and Lemma \ref{Lemma:Countingrt-2} has the following form:
\begin{lemma}\label{Lemma:Countingrt-2-Cor}
Let $(G,C)$ be an edge-colored graph on $n$ vertices with $\delta^c(G)$
and furthermore, $e(G)$ is minimal. Let $V(G)=\{v_1,v_2,\ldots,v_n\}$
such that
$d^{mon}_G(v_1)\geq d^{mon}_G(v_2)\geq \cdots \geq d^{mon}_G(v_n).$
Then for each $1\leq i\leq \delta^c(G)$,
\begin{align*}
rt(G;v_i)\geq \frac{1}{2}\left((d^{mon}_G(v_i)+d^c_{G}(v_i)-1)(\overline{\sigma}^c_2(G)-n)+(|Y_{v_i}|d^{mon}_G(v_i)-\sum_{a\in Y_{v_i}}d^{mon}_G(a))\right).
\end{align*}
\end{lemma}

\begin{prop}\label{Prop:2}
Let $(G,C)$ be an edge-colored graph on $n$ vertices with $\delta^c(G)$
and furthermore, $e(G)$ is minimal. Let $V(G)=\{v_1,v_2,\ldots,v_n\}$
such that
$d^{mon}_G(v_1)\geq d^{mon}_G(v_2)\geq \cdots \geq d^{mon}_G(v_n).$
Let $Y_{v_i}$ be defined as in Lemma \ref{Lemma:Countingrt-2} with $|Y_{v_i}|=\delta^c(G)-1$.
Then for each $1\leq k\leq \delta^c(G)-1$,
\begin{align*}
\sum_{i=1}^k(|Y_{v_i}|d^{mon}_G(v_i)-\sum_{a\in Y_{v_i}}d^{mon}_G(a))\geq\left(\delta^c\sum_{i=1}^{k}d^{mon}(v_i)-k\sum_{i=1}^{\delta^c}d^{mon}(v_i)\right)\geq 0.
\end{align*}
\end{prop}
\begin{proof}
Note that for $v_i\in V(G)$, $v_i\notin Y_{i}:=Y_{v_i}$.
Hence,
$$\sum_{a\in Y_{i}}d^{mon}(a)\leq \sum_{j=1}^{i-1}d^{mon}(v_j)+\sum_{j=i+1}^{\delta^c}d^{mon}(v_j).$$
Thus,
$$\begin{aligned}
&\sum\limits_{i=1}^{k}\sum\limits_{a\in Y_{i}}d^{mon}(a)\leq \sum\limits_{i=1}^{k} \left(\sum\limits_{j=1}^{i-1}d^{mon}(v_j)+\sum\limits_{j=i+1}^{\delta^c}d^{mon}(v_j)\right)=k\sum_{i=1}^{\delta^c}d^{mon}(v_i)-\sum_{i=1}^kd^{mon}(v_{i}).
\end{aligned}$$
It follows that
$$\begin{aligned}
\sum_{i=1}^{k}\left(d^{mon}(v_i)(\delta^c-1)-\sum_{a\in Y_{i}}d^{mon}(a)\right)\geq \left(\delta^c\sum_{i=1}^{k}d^{mon}(v_i)-k\sum_{i=1}^{\delta^c}d^{mon}(v_i)\right).
\end{aligned}$$
The proof of Proposition \ref{Prop:2} is complete.
\end{proof}
Now we can obtain the following theorem.
\begin{theorem}
Let $(G,C)$ be an edge-colored graph on $n$ vertices with $\delta^c(G)$
and furthermore, $e(G)$ is minimal. Let $V(G)=\{v_1,v_2,\cdots,v_n\}$
such that
$d^{mon}_G(v_1)\geq d^{mon}_G(v_2)\geq \ldots \geq d^{mon}_G(v_n).$
Then for each $1\leq k\leq \delta^c(G)-1$,
\begin{align*}
\sum_{i=1}^k rt(G;v_i)&\geq\frac{1}{2}\left(\sum_{i=1}^k d^{mon}_G(v_i)+k(\delta^c(G)-1)\right)(\overline{\sigma_2}^c(G)-n)+\frac{\Delta_k(G)}{2}.
\end{align*}
where
$$\Delta_k(G)=\left(\delta^c(G)\sum_{i=1}^{k}d^{mon}_G(v_i)-k\sum_{i=1}^{\delta^c(G)}d^{mon}_G(v_i)\right).$$
\end{theorem}

\begin{proof}
This theorem directly follows from Lemma \ref{Lemma:Countingrt-2-Cor} and Proposition \ref{Prop:2}.
\end{proof}

\section{Edge-colored friendship subgraphs}\label{Sec:4}

In this section, we shall prove a result slightly stronger than Theorem \ref{Thm:FriendshipSub}.
For a graph $G$, we denote by $\Delta^{mon}(G):=\max\{d^{mon}_G(v):v\in V(G)\}$.
\begin{theorem}\label{Thm:Friendshipgraph}
Let $k,n$ be positive integers, and $G$ be an edge-colored graph on
$n$ vertices with $n\geq 50k^2$ where $k\geq 2$,
and $\delta^c(G)\geq \frac{n}{2}+k-1$.
Let $v\in V(G)$ such that $d^{mon}_G(v)=\Delta^{mon}(G)$.
Then $G$ contains $k$ rainbow triangles sharing only the vertex $v$
as the center (i.e., the underly graph is $F_k$ with $v$ as its center).
\end{theorem}

The following result on Tur\'an
number of friendship graphs is well known.

\begin{theorem}[\cite{EFGG95}]\label{Thm:EFGS95}
For every $k\geq 1$ and every $n\geq 50k^2$,
if  a  graph $G$ of  order $n$ satisfies $e(G)>ex(n,F_k)$,
then $G$ contains a copy of a $k$-friendship graph, where
$ex(n,F_k)=\lfloor\frac{n^2}{4}\rfloor+k^2-k$ if $k$
is odd; and $ex(n,F_k)=\lfloor\frac{n^2}{4}\rfloor+k^2-\frac{3k}{2}$
if $k$ is even.
\end{theorem}
The \emph{matching number} of a graph $G$,
denoted by $\alpha'(G)$, is defined to be the maximum number of pairwise disjoint edges
in $G$. Our proof of Theorem \ref{Thm:Friendshipgraph} uses a
famous result on Tur\'an number of a matching
with given number of edges due to Erd\H{o}s and Gallai \cite{EG59}.
\begin{theorem}[\cite{EG59}]\label{Thm:EG59}
Let $G$ be a graph on $n$ vertices. If $\alpha'(G)\leq k$
then $e(G)\leq \max\{\binom{2k+1}{2}, \binom{n}{2}-\binom{n-k}{2}\}$.
\end{theorem}
We also need a special case of Lemma \ref{Lemma:Countingrt-5}.

\begin{lemma}\label{Lemma:Countingrt-5}
Let $(G,C)$ be an edge-colored graph on $n$ vertices with $\delta^c(G)$
such that $e(G)$ is minimal.
Then for a vertex $v\in V(G)$ with $d^{mon}_G(v)=\triangle^{mon}(G)$,
we have
\begin{align*}
rt(G;v)\geq \frac{1}{2}\left(\Delta^{mon}(G)+d^c_G(v)-1)(\delta^c(G)+d^c_G(v)-n)\right).
\end{align*}
\end{lemma}
\begin{proof}
Putting the vertex $v$ as one with $d^{mon}(v)=\Delta^{mon}(G)$,
and $Y_v\subset N(v)\backslash X_v$\footnote{Recall the definition of $X_v$ above the statement of Lemma \ref{Lemma:Countingrt-2}.}
(such that for each $u,u'\in Y_v$, we have $C(uv)\neq C(u'v)$)
and $|Y_v|=d^c(v)-1$ in Lemma \ref{Lemma:Countingrt-2},
from the fact $$|Y_v|d^{mon}(v)-\sum_{a\in Y_v}d^{mon}(a)\geq 0,$$
we obtain the lemma.
\end{proof}

\noindent
{\bf Proof of Theorem \ref{Thm:Friendshipgraph}.}
Without loss of generality, assume that $G$ is edge-minimal
subject to the condition $\delta^c\geq \frac{n}{2}+k-1$.
We prove the theorem by contradiction.
Choose $v\in V(G)$ such that $d^{mon}(v)=\Delta^{mon}(G)$.

If $\Delta^{mon}(G)=1$, then $G$ is properly-colored.
Note that $e(G)\geq \frac{\delta^cn}{2}\geq \frac{n^2}{4}+\frac{kn}{2}-\frac{n}{2},$
and $ex(n,F_{k})\leq \left\lfloor\frac{n^2}{4}\right\rfloor+k^2-\frac{3k}{2}$
when $n\geq 50k^2$ by Theorem \ref{Thm:EFGS95}.
When $n\geq 50k^2$, we have $$\frac{n^2}{4}+\frac{kn}{2}-\frac{n}{2}>\left\lfloor\frac{n^2}{4}\right\rfloor+k^2-\frac{3k}{2}$$
(recall $k\geq 2$), and so $G$ contains a properly-colored $F_k$, and hence $k$ rainbow triangles sharing
one common vertex. Next we assume that $\Delta^{mon}(G)\geq 2$.

By Lemma \ref{Lemma:Countingrt-5},
\begin{align}\label{aligh-Theorem14-1}
rt(G;v)&\geq \frac{1}{2}\left((d^{mon}(v)+d^c(v)-1)(\delta^c+d^c(v)-n)\right)\geq (k-1)(d^{mon}(v)+d^c(v)-1).
\end{align}
Consider the graph $G'=G[X_v\cup Y_v]$. Then $|G'|=d^{mon}(v)+d^c(v)-1\geq \frac{n}{2}+k$.
Notice that each edge in $G'$ corresponds to a rainbow triangle through the vertex $v$.
From (\ref{aligh-Theorem14-1}), we have that
\begin{align}\label{aligh-Theorem14-2}
e(G')\geq (k-1)(d^{mon}(v)+d^c(v)-1)\geq \frac{(k-1)(n+2k)}{2}.
\end{align}

Since $G$ contains no $k$ rainbow triangles sharing one common vertex,
$G'$ contains no matching of size $k$.
That is, $\alpha'(G')\leq k-1$. So by Theorem \ref{Thm:EG59},
\begin{align}\label{aligh-Theorem14-3}
e(G')\leq \max\left\{\binom{2k-1}{2},\binom{k-1}{2}+(k-1)(|G'|-k+1)\right\}.
\end{align}
By simple algebra, we have $\binom{2k-1}{2}<\frac{(k-1)(n+2k)}{2}$ when $n\geq 2k-3$.
Furthermore,
\begin{align*}
&(k-1)(d^{mon}(v)+d^c(v)-1)-\binom{k-1}{2}-(k-1)(|G'|-k+1)\\
&\geq -\binom{k-1}{2}+(k-1)^2>0
\end{align*}
Thus, (\ref{aligh-Theorem14-2}) contradicts (\ref{aligh-Theorem14-3})
since $n\geq 2k-3$. The proof is complete. $\hfill\blacksquare$

\section{Concluding remarks}
In this paper, we present a tight color degree condition (up to a constant) for $k$
rainbow triangles sharing one common vertex (when $k$ is a fixed integer),
and highly suspect the tight one is $\frac{n+1}{2}$ for $n=\Omega(k^2)$
(by considering Theorem \ref{Thm:EFGS95}).

Erd\H{o}s et al. \cite{EFGG95} conjectured Theorem \ref{Thm:EFGS95} holds
for $n\geq 4k$. If the answer to this conjecture is positive, then
Theorem \ref{Thm:FriendshipSub} can be improved to all graphs with order $n\geq 4k$.
On the other hand, maybe an answer to the following is positive.
\begin{prob}
Let $n,k$ be two positive integers. Let $(G,C)$ be an edge-colored
graph on $n$ vertices with $\delta^c(G)\geq \frac{n+1}{2}$. Does there exist
a constant $c$, such that if $n\geq ck$ then $G$ contains
a properly-colored $F_k$?
\end{prob}

Recall that $f(n):=\min\{rt(G):G\in \mathcal{G}^{*}_n\}$ (see Section \ref{Sec:1}).
We conclude this paper with the following more feasible problem.

\begin{prob}
Determine the value of $\lim\limits_{n\rightarrow \infty}\frac{f(n)}{n^2}$.
\end{prob}

\noindent {\bf Acknowledgements.}
The second author thanks Xiaozheng Chen and Ruonan Li for
checking the proof of Lemma \ref{Lemma:Countingrt-1}.
Xueliang Li was supported by NSFC (Nos.\
11871034 and 12131013). Bo Ning was supported by
NSFC (Nos.\ 11771141, 11971346). Yongtang Shi was partially
supported by the NSFC (No.\ 11922112), Natural Science
Foundation of Tianjin (Nos. 20JCZDJC00840, 20JCJQJC00090).
Shenggui Zhang was partially supported by NSFC (Nos. 11671320,
12071370, 12131013, U1803263).


\begin{thebibliography}{10}
\bibitem{ADH19}
R. Aharoni, M. DeVos, and R. Holzman,
Rainbow triangles and the Caccetta-H\"{a}ggkvist conjecture,
\emph{J. Graph Theory} {\bf 92} (2019), no. 4, 347--360.

\bibitem{BHLPVY17}
J. Balogh, P. Hu, B. Lidick\'{y}, F. Pfender, J. Volecfg, and M. Young,
Rainbow triangles in three-colored graphs,
\emph{J. Combin. Theory, Ser. B} {\bf 126} (2017), 83--113.

\bibitem{BM08}
J.A. Bondy, U.S.R. Murty, Graph theory,
Graduate Texts in Mathematics, 244. Springer,
New York, 2008. xii+651 pp. ISBN: 978-1-84628-969-9.

\bibitem{BLWZ05}
H.J. Broersma, X. Li, G. Woeginger, and S. Zhang,
Paths and cycles in colored graphs,
\emph{Australas. J.~Combin.} {\bf 31} (2005), 299--311.

\bibitem{CKRY16}
R. \v{C}ada, A. Kaneko, Z. Ryj\'{a}\v{c}ek, and K. Yoshimoto, Rainbow cycles in edge-colored graphs,
\emph{Discrete Math.} {\bf 339} (2016), no. 4, 1387--1392.

\bibitem{C18}
H. Chen, Long rainbow paths and rainbow cycles in edge colored graphs--A survey, \emph{App. Math.
Comput.} {\bf 317} (2018), no. 15, 187--192.

\bibitem{CMNO21}
A. Czygrinow, T. Molla, B. Nagle, and R. Oursler,
On odd rainbow cycles in edge-colored graphs,
\emph{European J. Combin.} 94 (2021), Paper No. 103316, 10 pp.

\bibitem{EM20}
S. Ehard, E. Mohr, Rainbow triangles and cliques in edge-colored graphs,
\emph{European J. Combin.} {\bf 84} (2020), 103037, 12 pp.

\bibitem{E55}
P. Erd\H{o}s, Some theorems on graphs,
\emph{Riv. Lemat.} {\bf 9} (1955), 13--17. (in Hebrew with English summary)

\bibitem{E62-1}
P. Erd\H{o}s,
On a theorem of Rademacher-Tur\'{a}n,
\emph{Illinois J. Math.} {\bf 6} (1962), 122--127.

\bibitem{E62-2}
P.Erd\H{o}s,
On the number of complete subgraphs contained in certain graphs,
\emph{Magy. Tud. Acad. Mat. Kut. Int. K\H{o}zl.}
{\bf 7} (1962),~459--474.

\bibitem{EFGG95}
P. Erd\H{o}s, Z. F\"{u}redi, R. J. Gould, and D. S. Gunderson,
Extremal graphs for intersecting triangles,
\emph{J. Combin. Theory. Ser. B} {\bf 64} (1995), no.1, 89--100.

\bibitem{EG59}
P. Erd\H{o}s, T. Gallai, On maximal paths and circuits of graphs,
\emph{Acta Math. Acad. Sci. Hungar} {\bf 10} (3--4) (1959) 337--356.

\bibitem{ESS73}
P. Erd\H{o}s, M. Simonovits, and V. T. S\'{o}s, Anti-Ramsey theorems, in Coll. Math. Soc. J. Bolyai 10, 1973, pp. 633-643.

\bibitem{FLZ18}
S. Fujita, R. Li, and S. Zhang,
Color degree and monochromatic degree conditions for short properly colored cycles in edge-colored graphs,
\emph{J. Graph Theory} {\bf 87} (2018), no.~3, 362--373.

\bibitem{FNXZ19}
S. Fujita, B. Ning, C. Xu, and S. Zhang,
On sufficient conditions for rainbow cycles in edge-colored graphs,
\emph{Discrete Math.} {\bf 342} (2019), no. 7, 1956--1965.

\bibitem{HLY20}
J. Hu, H. Li, and D. Yang, Vertex-disjoint rainbow triangles in edge-colored graphs,
\emph{Discrete Math.} {\bf 343} (2020), no. 12, 112117, 5 pp.


\bibitem{KL08}
M. Kano, X. Li, Monochromatic and heterochromatic subgraphs in edge-colored graphs¡ªa survey,
\emph{Graphs Combin.} {\bf 24} (2008), no. 4, 237--263.

\bibitem{LBZ20}
R. Li, H.J. Broersma, and S. Zhang,
Vertex-disjoint properly edge-colored cycles in edge-colored complete graphs,
\emph{J. Graph Theory} {\bf 94} (2020), no.~3, 476--493.


\bibitem{LNXZ14}
B. Li, B. Ning, C. Xu, and S. Zhang,
Rainbow triangles in edge-colored graphs,
\emph{European J. Combin.} {\bf 36} (2014), 453--459.


\bibitem{L13}
H. Li, Rainbow $C_3$'s and $C_4$'s in edge-colored graphs,
\emph{Discrete Math.} {\bf 313} (2013), no.~19, 1893--1896.


\bibitem{LW12}
H. Li, G. Wang,
Color degree and heterochromatic cycles in edge-colored graphs¡±,
\emph{European J. Combin.} {\bf 33} (2012), no.~8, 1958--1964.

\bibitem{LS83}
L. Lov\'{a}sz, M. Simonovits, On the Number of Complete Subgraphs of a Graph, II,
in: Studies in Pure Math, Birkh\"{a}user, 459--495, 1983.

\bibitem{XHWZ16}
C. Xu, X. Hu, W. Wang, and S. Zhang, Rainbow cliques in edge-colored graphs,
\emph{European J. Combin.} {\bf 54} (2016), 193--200.

\bibitem{XMZ20}
C. Xu, C. Magnant, and S. Zhang, Properly colored $C_4$'s in edge-colored graphs,
\emph{Discrete Math.} {\bf 343} (2020), no.~12, 112116, 13 pp.
\end{thebibliography}
\end{document}